\documentclass[10.9pt,a4paper]{amsart}
\usepackage[utf8]{inputenc}
\usepackage[english]{babel}
\usepackage{enumerate}

\newtheorem{theorem}{Theorem}[section]
\newtheorem{conj}{Conjecture}

\newtheorem{lemma}[theorem]{Lemma}
\newtheorem{proposition}[theorem]{Proposition}

\newtheorem{problem}{Problem}

\theoremstyle{definition}
\newtheorem{definition}[theorem]{Definition}

\newtheorem{example}[theorem]{Example}
\newtheorem{remark}[theorem]{Remark}

\newtheorem*{solution}{Solution}
 
\theoremstyle{remark}
\newtheorem*{ack}{Acknowledgments}

\usepackage{todonotes}
\usepackage[square,sort,comma,numbers]{natbib}

\usepackage{standalone}
\usepackage[colorlinks,
linkcolor={red!50!black},
citecolor={blue!50!black},
urlcolor={blue!80!black}]{hyperref}

\usepackage{algorithm}
\usepackage{algpseudocode}

\usepackage{url}
\usepackage[all]{xy}
\usepackage{enumerate}
\usepackage{amsfonts,amssymb,amsmath,pinlabel,array,hhline}
\usepackage{slashed}
\usepackage{tabulary}
\usepackage{fancyhdr}
\usepackage{a4wide}
\usepackage{bbm,dsfont}
\usepackage[position=b]{subcaption}
\usepackage{graphicx}
\usepackage{comment}

\usepackage[inkscapearea=page]{svg}

\newcommand{\Z}{\mathds{Z}}
\newcommand{\Q}{\mathds{Q}}

\newcommand{\C}{\mathds{C}}

\newcommand{\sign}{\operatorname{sign}}

\DeclareMathOperator{\Char}{Char}

\renewcommand{\setminus}{\smallsetminus}

\DeclareSymbolFont{EulerScript}{U}{eus}{m}{n}
\DeclareSymbolFontAlphabet\mathscr{EulerScript}

\title{On the slice genus of generalized algebraic knots}
\author{Maria Marchwicka}
\address{Department of Mathematics and Computer Science, Adam Mickiewicz University in Poznań, ul. Uniwersytetu Poznańskiego 4, 61-614 Poznań, Poland}
\email{maria.marchwicka@amu.edu.pl}
\author{Wojciech Politarczyk}
\address{Institute of Mathematics, University of Warsaw, ul. Banacha 2, 02-097 Warsaw, Poland}
\email{wpolitarczyk@mimuw.edu.pl}

\begin{document}
\maketitle{}

\begin{abstract}
We give examples of a linear combination of algebraic knots and their mirrors that are algebraically slice,
but whose topological and smooth four-genus is two. Our examples generalize an example of non-slice algebraically slice
linear combination of iterated torus knots obtained by Hedden, Kirk and Livingston.
Our main tool is a genus bound from Casson--Gordon invariants and a cabling formula that allows us
to compute effectively these invariants.
\end{abstract}

\section{Introduction}
\label{sec:introduction}

A knot \(K \subset S^{3}\) is \emph{algebraic} if it arises as a link of an isolated singularity of a complex curve.
A \emph{generalized algebraic knot} (\emph{GA-knot} for short) is a knot that can be written as a connected sum of algebraic knots and their reverse mirror images.

Rudolph asked~\cite{Rudolph} whether algebraic knots are linearly independent in the smooth concordance group.
By~\cite[Corollary 8.4]{Miyazaki}, Rudolph's question can be reformulated by asking whether every smoothly slice GA-knot is ribbon.
For the sake of brevity, we will refer to the conjecture asserting the affirmative answer to Rudolph's question as the~\emph{Rudolph's conjecture}.

Probably the first result towards Rudolph's conjecture is due to Litherland~\cite{Litherland-signature}.
By studying the jumps of the Levine-Tristram signature function, he proved that torus knots are linearly independent in the smooth concordance group.
Livingston and Melvin~\cite{LivingstonMelvinAlgebraicKnots} proved that there are GA-knots that are algebraically slice but not smoothly slice, hence the Levine-Tristram signature is not sufficient to answer Rudolph's question.
Later, Hedden, Kirk, and Livingston~\cite{HeddenKirkLivingston} used Casson-Gordon invariants to show that there is an infinite family of algebraically-slice GA-knots, which are not smoothly slice.
Their approach was extended by Conway, Kim, and the second author~\cite{conway2019nonslice}.

Recall that in dimension four, there is a big difference between smooth and topological (locally flat) concordance of knots and links.
Using~\cite{Nagel-Powell}, it can be verified that Litherland's result holds also in the topological concordance group.
Similarly, it's been understood for many years that the invariance of Casson-Gordon signatures with respect to topological concordance can be deduced from the results of Freedman-Quinn~\cite{Freedman-Quinn} and Gilmer~\cite{gilmerTopologicalProofSignature1981}.
Consequently, the independence results discussed in the previous paragraph hold in the topological concordance group as well.


Examples of algebraically slice non-slice GA-knots studied in~\cite[Section 8]{HeddenKirkLivingston} have topological four-genus equal to one.
However, the main result of~\cite{conway2019nonslice} gives infinitely generated subgroups of the topological concordance group consisting of algebraically-slice GA-knots.
Therefore, it is an interesting question whether there are examples of algebraically slice GA-knots with arbitrary large topological four-genus.
We propose the following conjecture.
\begin{conj}\label{conj:main-conjecture}
  There are algebraically slice GA-knots with arbitrary large topological four-genus.
\end{conj}

The purpose of this paper is to provide some examples supporting Conjecture~\ref{conj:main-conjecture}.

Before stating our main theorem, let us introduce some notation.
For the rest of the paper, for positive coprime integers $p,q$, we let \(T(p,q)\) denote the \((p,q)\)-torus knot. 
If also $r,s$ are positive coprime integers, we write \(T(p,q;r,s)\) for the \((r,s)\)-cable of \(T(p,q)\). 
Note that every positive torus knot is an algebraic knot.
Recall that if $s>pqr$, the knot $T(p,q;r,s)$ is algebraic as well; see~\cite{EN}.
For a knot \(K\), \(g_{4}(K)\) denotes the smooth four-genus of \(K\) and \(g_{4}^{top}(K)\) denotes the topological four-genus of \(K\).
Finally, let \(-K\) denote the reverse mirror image of \(K\).

\begin{theorem}\label{thm:main-theorem}
  Consider the following GA-knot
  \begin{align*}
    K &= T(2, 17; 2, 83) \ \# \ -T(2, 11; 2, 83) \ \# \ T(2, 83) \ \# \ -T(2, 13; 2, 83) \\
      &\# \ T(2, 11; 2, 103) \ \# \ -T(2, 103) \  \# \ T(2, 13; 2, 103) \ \# \ -T(2, 17; 2, 103).
  \end{align*}
  \(K\) is algebraically slice and \(g_{4}(K) = g_{4}^{top}(K) = 2\).
\end{theorem}

There are many more examples of algebraically slice GA-knots with \(g_{4} = g_{4}^{top} = 2\).
Among them, the example from Theorem~\ref{thm:main-theorem} is the smallest one we could find.
Few other examples are listed below.
\begin{align*}
K_1 &= T(2, 23; 2, 107) \ \# \  -T(2, 17; 2, 107) \ \# \  T(2, 107) \ \# \  -T(2, 19; 2, 107) \\ &\# \  T(2, 17; 2, 131) \ \# \  -T(2, 131) \ \# \  T(2, 19; 2, 131) \ \# \  -T(2, 23; 2, 131)
\\
K_2 &= T(2, 29; 2, 139) \ \# \  -T(2, 19; 2, 139) \ \# \  T(2, 139) \ \# \  -T(2, 23; 2, 139) \\ &\# \  T(2, 19; 2, 163) \ \# \  -T(2, 163) \ \# \  T(2, 23; 2, 163) \ \# \  -T(2, 29; 2, 163)
\\
K_3 &= T(2, 37; 2, 163) \ \# \  -T(2, 29; 2, 163) \ \# \  T(2, 163) \ \# \  -T(2, 31; 2, 163) \\ &\# \  T(2, 29; 2, 181) \ \# \  -T(2, 181) \ \# \  T(2, 31; 2, 181) \ \# \  -T(2, 37; 2, 181)
\end{align*}
To keep the paper concise we focus on a single example given in Theorem~\ref{thm:main-theorem}.

Our main tool for establishing the necessary genus bounds is the Casson-Gordon \(\sigma\)-invariant~\cite{CassonGordon1,CassonGordon2}.
In particular, we use the results of Gilmer~\cite{gilmerSliceGenusKnots1982a}, see also~\cite{FlorensGilmer}.
We note that the following tools are insufficient to give any nontrivial genus bounds.
\begin{itemize}
\item By the definition, Levine-Tristram signatures fail to provide any genus bounds for an algebraically slice knot.
\item For an algebraically-slice GA-knot~\(K\) we have $\tau(K)=s(K)=0$.
  Similarly, all the higher $s$-invariants coming from $sl(N)$ homologies~\cite{Lewar-Lobb,Lewark}, for \(N \geq 2\), are zero.
  This is proved in~\cite[Proposition 8.2]{HeddenKirkLivingston}.
\item The $\Upsilon$ invariant of Ozsv\'ath, Stipsicz, and Szab\'o~\cite{OSS} vanishes for $K$. The argument is as in \cite[Section 1.2]{conway2019nonslice} and involves the result of Tange~\cite{Tange}.
\end{itemize}

The proof of Theorem~\ref{thm:main-theorem} consists of two steps: establishing an upper and a lower bound.
To obtain the upper bound, we construct an explicit genus-two surface in~\(B^{4}\) which bounds the knot from Theorem~\ref{thm:main-theorem}.
This is done in Lemma~\ref{lemma:upper-bound}.
On the other hand, the lower bound is established with the aid of Gilmer's bound~\cite[Theorem 1]{gilmerSliceGenusKnots1982a} and hence requires much more work.
In order to apply Gilmer's lower bound, we had to resort to computer calculations to verify the required assumptions.
The calculations were done using a SageMath~\cite{sagemath} script, available at~\cite{sagecode}.
The output of the Sage script is attached to the arXiv version of the paper.

The complexity of calculations required to verify assumptions of the Gilmer's lower bound increases significantly if one takes more complicated examples of GA-knots.
However, it is possible to make some optimizations by considering only the primary summands of the linking form, see Remark~\ref{remark:optimizations} for a more detailed discussion.
In particular, for the knot from Theorem~\ref{thm:main-theorem}, we were able to reduce the number of cases to check from roughly~\(6 \cdot 10^{12}\) to~\(1.6 \cdot 10^{6}\).
We made attempts to find examples of algebraically slice GA-knots with topological four-genus equal to three, however, the calculations were too complicated even including the aforementioned optimizations.

The paper is organized as follows.
Section~\ref{sec:casson-gordon-invariants} serves as a quick recollection of the most relevant properties of Casson-Gordon \(\sigma\)-invariants and the four-genus bound of Gilmer.
In Section~\ref{sec:computations} we discuss our implementation of calculations of Casson-Gordon invariants.
In Section~\ref{sec:genus-bounds}, we give the proof of Theorem~\ref{thm:main-theorem}.

\begin{ack}
  The authors would like to thank Maciej Borodzik for helpful discussions and for careful reading of the
  preliminary version of the paper.
  The second author is grateful to Anthony Conway for many interesting conversations that inspired the paper.
  The authors are also grateful to Lukas Lewark for pointing out some errors in the first version of the paper.
  We are also grateful to the anonymous referee whose insightful comments helped us to greatly improve the paper.

  
  The first author was supported by the National Science Center grant 
  2016/22/E/ST1/00040.
\end{ack}

\section{Casson-Gordon invariants}
\label{sec:casson-gordon-invariants}

\subsection{Linking forms and characters}
\label{sec:link-forms-char}
In this section we review some basic properties of linking forms.
For a more detailed study of the algebra of linking forms refer to~\cite{KawauchiKojima,WallQuadratic}.

For an integer \(m>1\), let \(C_{m} \subset S^{1}\) denote the multiplicative subgroup generated by primitive roots of unity of order \(m\).
Similarly, by \(C_{\infty}\) we denote the subgroup of \(S^{1}\) consisting of roots of unity, i.e., \(C_{\infty}\) is generated by the set \(\bigcup_{m \geq 1} C_{m}\).

If \(A\) is a finite abelian group, a \emph{character} on~\(A\) is a homomorphism
\[\chi \colon A \to C_{\infty}.\]
For a fixed \(m \geq 1\), we say that the character~\(\chi\) is an \emph{order~\(m\) character}, if \(\chi(A) = C_{m}\).
Observe that the set~\(\Char(A)\) of characters on a finite abelian group~\(A\) can be identified with the set \(\hom(A,\Q / \Z)\).
Indeed, any \(\phi \in \hom(A, \Q / \Z)\) corresponds to the character given by the formula
\begin{equation}
  \label{eq:correspondence-characters-homomorphism}
  \chi(\phi)(y) = \exp \left(  2 \pi i \phi(y) \right).
\end{equation}

A \emph{linking form} is a pair \((A,\lambda)\) consisting of a finite abelian group \(A\) and a nonsingular symmetric
bilinear pairing
\[\lambda \colon A \times A \to \Q / \Z.\]
Here, nonsingularity means that the adjoint map
\[\lambda^{D} \colon A \to \hom(A,\Q / \Z), \quad \lambda^{D} \colon x \mapsto \left( y \mapsto \lambda(x,y) \right)\]
is an isomorphism of abelian groups.

\begin{example}\label{example:linking-forms-cyclic-groups}
  Fix a prime \(q \in \Z\).  For a non-zero integer \(a\) we will denote by \((a/q)\) the linking form
  \[(a/q) \colon \Z_{q} \times \Z_{q} \to \Q / \Z, \quad (x,y) \mapsto \frac{a}{q} \cdot x \cdot y.\]
  Observe that \((a/q)\) is nonsingular if \(\gcd(a,q) = 1\).
\end{example}

\begin{example}
  Let \(K\) be a knot in \(S^{3}\) and let \(\Sigma(K)\) denote the double branched cover of \(K\).
  By~\cite[Proposition 8.20]{Burde-Zieschang}, \(\Sigma(K)\) is a rational homology sphere, hence there is a linking form
  \[\lambda_{K} \colon H_{1}(\Sigma(K)) \times H_{1}(\Sigma(K)) \to \Q / \Z.\]
  According to~\cite[Section 2.2]{Conway-Friedl-Gerrit}, the adjoint of \(\lambda_{K}\) is equal to the composition
  \[\lambda_{K}^{D} \colon H_{1}(\Sigma(K)) \xrightarrow{PD} H^{2}(\Sigma(K)) \xrightarrow{B} H^{1}(\Sigma(K);\Q / \Z) \xrightarrow{\operatorname{ev}} \hom(H_{1}(\Sigma(K)),\Q / \Z).\]
  In the above diagram, \(PD\) denotes the inverse of the Poincar\'e duality isomorphism given by the cap product with the fundamental class of \(\Sigma(K)\), \(B\) denotes the inverse of the Bockstein map associated to the short exact sequence \(0 \to \Z \to \Q \to \Q / \Z \to 0\), and
  \(\operatorname{ev}\) is the evaluation map coming from the Universal Coefficient Theorem.
  Since \(\Sigma(K)\) is a rational homology sphere, the maps~\(B\) and~\(\operatorname{ev}\) are isomorphisms.
  Consequently, \(\lambda_{K}\) is indeed nonsingular.
\end{example}

For two linking forms \((A_{1},\lambda_{1})\) and \((A_{2},\lambda_{2})\), we define their (orthogonal) sum \((A_{1}\oplus A_{2},\lambda_{1} \oplus \lambda_{2})\) by the formula
\[(\lambda_{1} \oplus \lambda_{2})((x_{1},x_{2}),(y_{1},y_{2})) = \lambda_{1}(x_{1},y_{1}) + \lambda_{2}(x_{2},y_{2}).\]
For \(n>0\), we will denote by \(n \lambda_{1}\) the \(n\)-fold direct sum of the linking form \(\lambda_{1}\).
For a subgroup \(M \subset A\) define its \emph{orthogonal complement}
\[M^{\perp} = \left\{x \in A \colon \forall_{y \in M} \quad \lambda(x,y) = 0\right\}.\]
We say that \(M \subset A\) is~\emph{isotropic} if \(M \subset M^{\perp}\) and a is~\emph{metabolizer} if \(M^{\perp} = M\).
Similarly, we say that \(x \in M\) is~\emph{isotropic}, if~\(\lambda(x,x)=0\).
The linking form \((A,\lambda)\) is~\emph{metabolic} if it admits a metabolizer.

For a symmetric integral \(n \times n\) matrix \(H\) with \(\det(H) \neq 0\), consider the symmetric bilinear form
\[\lambda_{H} \colon \left(\Z^{n} / H \Z^{n}\right) \times \left(\Z^{n} / H \Z^{n}\right) \to \Q / \Z, \quad \lambda_{H}(x + H \Z^{n},y + H \Z^{n}) = x^{T} \cdot H^{-1} \cdot y.\]
We say that the linking form \((A,\lambda)\) is \emph{represented} by the symmetric integral matrix \(H\), if \((A,\lambda)\) is isomorphic to \((\Z^{n} / H \Z^{n}, \lambda_{H})\).

If \((A,\lambda)\) is a linking form, the underlying abelian group \(A\) admits a primary decomposition
\[A = \bigoplus_{p} A_{p},\]
where the above direct sum is taken over all primes and \(A_{p}\) denotes the \(p\)-primary part of \(A\).
According to~\cite[Section 5]{WallQuadratic}, the primary decomposition of \(A\) gives a primary decomposition of \(\lambda\)
\begin{equation}
  \label{eq:primary-decomposition-linking-form}
  (A, \lambda) = \bigoplus_{p} (A_{p},\lambda_{p}),
\end{equation}
where \(\lambda_{p}\) denotes the restriction of \(\lambda\) to \(A_{p}\).

\begin{lemma}\label{lemma:primary-decomposition-metabolizers}
  Suppose that \((A,\lambda)\) is a linking form admitting a metabolizer \(M\), then \(M\) admits a primary decomposition
  \[M = \bigoplus_{p} M_{p},\]
  where \(M_{p} = M \cap A_{p}\) is a metabolizer for the linking form \((A_{p},\lambda_{p})\), for any prime \(p\).
\end{lemma}
\begin{proof}
  Choose a metabolizer \(M \subset A\) for \(\lambda\) and let \(M_{p} = M \cap A_{p}\), for any prime \(p\).
  Observe that \(M_{p}^{\perp_{\lambda_{p}}} = M^{\perp_{\lambda}} \cap A_{p}\), where \(M^{\perp_{\lambda}}\) and \(M^{\perp_{\lambda_{p}}}\) denote orthogonal complements with respect to \(\lambda\) and
  \(\lambda_{p}\), respectively.
  Indeed, the inclusion \(M^{\perp_{\lambda}} \cap A_{p} \subset M_{p}^{\perp_{\lambda_{p}}}\) is obvious.
  On the other hand, if \(x \in A_{p}\), then by the primary decomposition~\eqref{eq:primary-decomposition-linking-form}, we have \(\lambda(x,y) = 0\), for any \(y \in A_{q}\), for any prime \(q \neq p\).
  Hence \(M_{p}^{\perp_{\lambda_{p}}} \subset M^{\perp_{\lambda}} \cap A_{p}\).
  Therefore,
  \[M_{p}^{\perp_{\lambda_{p}}} = M^{\perp_{\lambda}} \cap A_{p} = M \cap A_{p} = M_{p},\]
  which proves that \(M_{p}\) is a metabolizer for \(\lambda_{p}\), as desired.
\end{proof}

Let us finish this section with the following lemma relating linking forms and characters.

\begin{lemma}\label{lemma:character-vector-correspondence}
  Let \((A,\lambda)\) be a linking form, then \(\lambda\) determines a natural identification of abelian groups
  \[T_{\lambda} \colon A \cong \Char(A),\]
  where, for \(x \in A\), \(T_{\lambda}(x) = \chi_{x}\), where \(\chi_{x}(y) = \exp(2 \pi i \lambda(x,y)))\).
\end{lemma}
\begin{proof}
  As we observed above, there is an identification \(\hom(A,\Q / \Z) \cong \Char(A)\) given by formula~\eqref{eq:correspondence-characters-homomorphism}.
  Furthermore, by nonsingularity of \(\lambda\), we obtain a natural identification \(\lambda^{D} \colon A \cong \hom(A,\Q / \Z)\), hence the lemma follows.
  We can define \(T_{\lambda}\) to be the composition of these isomorphisms.
\end{proof}





\subsection{Background on Casson-Gordon invariants}
\label{sec:background-casson-gordon-inv}

Fix a positive integer \(m>0\).
Recall from the previous section that \(C_{m}\) denotes the subgroup of \(S^{1}\) generated by primitive roots of unity of order \(m\).
A topological \emph{\(n\)-manifold over \(C_{m}\)} (or a topological \(n\)-dimensional \(C_{m}\)-manifold) is a pair \((M,\chi)\), where \(M\) is a compact, oriented topological \(n\)-manifold and \(\chi \colon H_{1}(M) \to C_{m}\) is a homomorphism.

We say that a \(C_{m}\)-manifold \((M,\chi)\) of dimension~\(n\) \emph{bounds over \(C_{m}\)} if there exists a topological \((n+1)\)-manifold \((W,\psi)\) over \(C_{m}\) such that
\[\partial (W, \psi) = (M,\chi).\]
In particular, we require that \(M\) is the oriented boundary of \(W\) and \(\chi = \psi \circ \iota\), where \(\iota \colon M \hookrightarrow W\) is the inclusion map.
Furthermore, we say that two closed topological \(n\)-manifolds \((M_{1},\chi_{1})\), \((M_{2},\chi)\) over \(C_{m}\) are \emph{bordant over \(C_{m}\)} if the \(C_{m}\)-manifold \(\left(M_{1} \sqcup -M_{2}, \chi_{1} \oplus
    \chi_{2}\right)\) bounds over \(C_{m}\).
For a positive integer \(r\), we will use the abbreviated notation
\[r \cdot (M,\chi) = \left( \bigsqcup_{r} M, \bigoplus_{r} \chi \right).\]

The following lemma is crucial for defining Casson-Gordon invariants of knots and links in the topological category.
The lemma is known to the experts in the field, however, in the literature, the lemma is stated for smooth manifolds only, see~\cite[Proposition 2.12]{conway-CG-notes} or~\cite[p. 183 below the example]{CassonGordon2}.
Hence, to fill this gap, we provide a sketch of the proof which works in the topological category.

\begin{lemma}\label{lemma:topological-bordism}
  Fix an integer \(m>0\).
  Let \((M,\chi)\) be a closed \(3\)-dimensional \(C_{m}\)-manifold, then there exists an integer \(r>0\) and a \(4\)-dimensional \(C_{m}\)-manifold \((W,\psi)\) such that
  \[r \cdot (M,\chi) = \partial (W,\psi),\]
  i.e., \(r \cdot (M,\chi)\) bounds over \(C_{m}\).
\end{lemma}
\begin{proof}[Sketch of the proof]
  We only sketch the argument.
  More details can be found in~\cite[Chapters 10-11 of Annex C]{kirby-siebenmann},~\cite{kupers}, see also~\cite{stong} or~\cite{davis-kirk}.

  Recall that \(\Omega_{3}^{STOP}(C_{m})\) denotes the bordism group of closed oriented topological \(3\)-dimensional \(C_{m}\)-manifolds.
  In other words, \(\Omega_{3}^{STOP}(C_{m})\) is the set of equivalence classes of homeomorphisms classes of closed
  oriented \(3\)-dimensional \(C_{m}\)-manifolds under the relation of bordism over \(C_{m}\).
  Disjoint sum equips \(\Omega_{3}^{STOP}(C_{m})\) with the structure of an abelian group (the empty manifold is the neutral element).
  In particular, a \(C_{m}\)-manifold \(r \cdot (M,\chi)\), where \(r>0\), represents the trivial element in \(\Omega_{3}^{STOP}(C_{m})\) if it bounds over \(C_{m}\).
  Consequently, it is sufficient to prove that~\(\Omega_{3}^{STOP}(C_{m})\) is a finite abelian group.
  
  By topological transversality, see~\cite{kirby-siebenmann},~\cite[Theorem 9.5A]{Freedman-Quinn} or~\cite[Chapter 10]{friedl2020survey}, the Pontriagin-Thom construction works in the topological category,
  hence there is an isomorphism of abelian groups, see~\cite[Lecture 3]{kupers},
  \[\Omega_{n}^{\operatorname{STOP}}(C_{m}) \cong \pi_{n}^{S}(BC_{m} \wedge \operatorname{MSTOP}),\]
  where \(BC_{m}\) is the classifying space of the group \(C_{m}\) and \(\operatorname{MSTOP}\) denotes the relevant Thom spectrum, see~\cite[Chapter 8.7]{davis-kirk}.
  Stable homotopy groups \(\pi_{n}^{S}(BC_{m} \wedge \operatorname{MSTOP})\) can be calculated with the aid of the Atiyah-Hirzebruch spectral sequence, refer to~\cite[Section 9.2]{davis-kirk}.
  Since
  \[\pi_{n}^{S}(\operatorname{MSTOP}) = \Omega_{n}^{\operatorname{STOP}}(pt) =
    \begin{cases}
      \Z, & n = 0, \\
      0, & n=1,2,3,
    \end{cases}
  \]
  see~\cite[Annex C, Theorem 11.1]{kirby-siebenmann}, the Atiyah-Hirzebruch spectral sequence yields an isomorphism of abelian groups
  \[\Omega_{3}^{\operatorname{STOP}}(C_{m}) \cong \pi_{3}^{S}(BC_{m} \wedge \operatorname{MSTOP}) \cong H_{3}(BC_{m}) \cong C_{m},\]
  compare with computations of oriented bordims groups in~\cite[Section 9.3]{davis-kirk} just after Theorem~9.10.
  Above, the first isomorphism follows from the topological version of the Pontriagin-Thom construction, the second isomorphism follows from the Atiyah-Hirzebruch spectral sequence, and the last isomorphism
  follows from formula~(3.1) in~\cite[Chapter II.3]{brown}.
\end{proof}

Let \((W,\psi)\) be a \(4\)-dimensional topological \(C_{m}\)-manifold with boundary.
The cyclic group \(C_{m}\) acts by multiplication on the complex plane, hence we can consider twisted homology groups \(H_{\ast}(W,\C^{\psi})\) and \(H_{\ast}(\partial W,\C^{\chi})\), where \(\chi\)
denotes the restriction of \(\psi\) to \(\partial W\).
Recall that the \(\psi\)-twisted Poincar\'e duality can be used to construct the~\(\psi\)-twisted intersection form of \(W\)
\[Q_{\psi} \colon H_{2}(W,\C^{\psi}) \times H_{2}(W,\C^{\psi}) \to \C,\]
see~\cite[Appendix D.6]{viro}.
The \(\psi\)-twisted intersection form is hermitian with respect to the complex conjugation.
We will denote by \(\sign^{\psi}(W)\) the signature of \(Q_{\psi}\).
Similarly, we will denote by \(\sign(W)\) the untwisted signature of \(W\), i.e., \(\sign (W) := \sign^{\psi_{0}}(W)\), where \(\psi_{0}\) is the trivial homomorphism.



\begin{definition}[\cite{CassonGordon2,CassonGordon1}]\label{definition:CS-invariants}
  Fix a knot \(K\) and a character \(\chi \colon H_{1}(\Sigma(K)) \to C_{m}\).
  The \emph{Casson-Gordon \(\sigma\)-invariant} and the \emph{Casson-Gordon nullity} of \(K\) are defined by the formulas
  \begin{align*}
    \sigma(K,\chi) &= \frac{1}{r} \left( \sign^{\psi}(W) - \sign(W) \right), \\
    \eta(K,\chi) &= \dim_{\C} H_{1}(\Sigma(K),\C^{\chi}),
  \end{align*}
  where \((W,\psi)\) is a \(4\)-dimensional \(C_{m}\)-manifold bounding \(r \cdot (\Sigma(K),\chi)\), for some positive integer \(r\).
\end{definition}

\begin{remark}
  Observe that Lemma~\ref{lemma:topological-bordism} guarantees the existence of the required four-dimensional \(C_{m}\)-manifold in Definition~\ref{definition:CS-invariants}.
\end{remark}

Basic properties of the \(\sigma\)-invariant and nullity are summarized in the following proposition.

\begin{proposition}\label{prop:properties-cs-invariants}
  Let \(K\) be a knot in \(S^{3}\).
  \begin{enumerate}
  \item If \(\chi\) is trivial, then \label{item:properties-cs-invariants-1}
    \[\sigma(K,\chi) = \eta(K,\chi) = 0.\]
  \item If \(K_{1}\) and \(K_{2}\) are knots with characters \(\chi_{1}\) and \(\chi_{2}\), then
    \begin{align*}
      \sigma(K_{1} \# K_{2},\chi_{1} \oplus \chi_{2}) &= \sigma(K_{1},\chi_{1}) + \sigma(K_{2},\chi_{2}), \\
      \eta(K_{1} \# K_{2},\chi_{1} \oplus \chi_{2}) &=
                            \begin{cases}
                              \eta(K_{1},\chi_{1}) + \eta(K_{2},\chi_{2}) + 1, & \chi_{1},\chi_{2} \text{ nontrivial}, \\
                              \eta(K_{1},\chi_{1}) + \eta(K_{2},\chi_{2}), & \text{otherwise}.
                            \end{cases}
    \end{align*}
    \label{item:properties-cs-invariants-2}
  \end{enumerate}
\end{proposition}
\begin{proof}
  Let us first prove Item~\ref{item:properties-cs-invariants-1}.
  Suppose that \(\chi\) is the trivial character on \(H_{1}(\Sigma(K))\).
  Let \(W\) be a compact, connected, oriented topological \(4\)-manifold bounding \(\Sigma(K)\).
  Since the \(3\)-dimensional topological oriented bordism group \(\Omega_{3}^{STOP}(pt)\) is trivial, such \(W\) always exists.
  For the proof of triviality of~\(\Omega_{3}^{STOP}(pt)\) refer to~\cite[Annex C, Theorem 11.1]{kirby-siebenmann}.

  Let \(\psi_{0} \colon H_{1}(W) \to C_{m}\) be the trivial character.
  Since \(\chi\) is trivial, \(\partial(W,\psi_{0}) = (\Sigma(K),\chi)\) in \(\Omega_{3}^{STOP}(C_{m})\).
  By definition,
  \[\sigma(K,\chi) = \sign^{\psi_{0}}(W) - \sign(W) = \sign(W) - \sign(W) = 0,\]
  since the triviality of \(\psi_{0}\) implies that \(\sign^{\psi_{0}}(W) = \sign(W)\).

  Similarly,
  \[\eta(K,\chi) = \dim_{\C}H_{1}(\Sigma(K),\C^{\chi}) = \dim_{\C}(\Sigma(K);\C) = 0.\]
  The middle equality above follows from the fact that \(\chi\) is trivial.
  The last equality follows from the fact that if \(K\) is a knot, then \(\Sigma(K)\) is a rational homology sphere, see~\cite[Proposition 8.20]{Burde-Zieschang}.

  For the proof of~Item~\ref{item:properties-cs-invariants-2} refer to~\cite[Proposition 2.5]{FlorensGilmer}. 
\end{proof}


The next theorem describes a method for obtaining genus bounds with the aid of the Casson-Gordon invariants.

\begin{theorem}[\cite{FlorensGilmer,gilmerSliceGenusKnots1982a}]\label{thm:genus-bounds}
  Let \(K \subset S^{3}\) be a knot and denote by \(\sigma_{K}\) the signature of \(K\).
  Suppose that \(K\) bounds a locally flat embedded surface of genus \(g\) in the four-ball.
  Then, there exists a decomposition of the linking form \(\lambda_{K} = \beta_{1} \oplus \beta_{2}\) such that
  \begin{enumerate}
  \item \(\beta_{1}\) admits a presentation by a matrix of rank \(2g\) and signature \(\sigma_{K}\),
  \item \(\beta_{2}\) is metabolic and there exists a metabolizer \(L\) of \(\beta_{2}\) such that for any \(x \in L \setminus \{0\}\) we have
    \[\left| \sigma(K,\chi_{x}) + \sigma_{K} \right| \leq \eta(K,\chi_{x}) + 4g+1.\]
  \end{enumerate}
\end{theorem}

\begin{remark}
  As mentioned in the introduction, Theorem~\ref{thm:genus-bounds} was originally formulated in the smooth category.
  However, it has been understood for many years that the results of Gilmer~\cite{gilmerTopologicalProofSignature1981} and Freedman-Quinn~\cite{Freedman-Quinn} can be used to adapt the original proof of
  Theorem~\ref{thm:genus-bounds} to work in the 
  topological category.
\end{remark}

\subsection{Casson-Gordon invariants of \((2,q)\)-cable knots}
\label{sec:torus-knots}


For a torus knot \(T(2,q)\), where \(q>1\) is odd, the double branched cover \(\Sigma(T(2,q))\) is the lens space \(L(q,1)\).
Indeed, the torus knot \(T(2,q)\) is the two-bridge knot \(b(q,1)\), see~\cite[Chapter 12.A]{Burde-Zieschang}, hence by~\cite[Proposition 12.3]{Burde-Zieschang}, \(\Sigma(T(2,q)) = \Sigma(b(q,1)) = L(q,1)\).
In particular, \(H_{1}(\Sigma(T(2,q))) \cong \Z_{q}\) and, by~\cite[Proposition 1.2]{Conway-Friedl-Gerrit}, the linking form \(\lambda_{T(2,q)}\) is isomorphic to the linking form \((-1/q)\),
see Example~\ref{example:linking-forms-cyclic-groups} for the notation.

For a general knot \(K\), the double branched cover of its \((2,q)\)-cable \(K(2,q)\), can be built form \(\Sigma(T(2,q))\), see~\cite[Lemma~4]{Litherland}.
In particular, the following lemma, taken from~\cite[Lemma 2.2]{HeddenKirkLivingston}, describes the most relevant properties of double-branched covers of \((2,q)\)-cable knots.

\begin{lemma}\label{lemma:branched-covers-cable-knots}
  Let \(K\) be a knot in \(S^{3}\), \(K(2,q)\) its \((2, q)\)-cable, and \(T(2,q)\) the \((2, q)\)-torus knot.
  Choose an odd prime \(p\) and let \(\xi_{p} = e^{2 \pi i / p}\).
  Then:
  \begin{itemize}
      \item There is a map \(r \colon \Sigma(K(2,q)) \to \Sigma(T(2,q)) = L(q,1)\), which is a homology isomorphism.
    Consequently, we can identify \(H_{1}(\Sigma(K(2,q))) \cong \Z_{q}\).
      \item Furthermore, the map \(r\) induces an isometry of linking forms.
    In particular, there is a generator \(l \in H_{1}(\Sigma(K(2,q)))\) such that \(\lambda_{K(2,q)}(l,l)=-1/q\).
  \end{itemize}
\end{lemma}

The following lemma summarizes computations of Casson-Gordon invariants of \((2,q)\)-cable knots for~\(q\) an odd prime.

\begin{lemma}\label{lemma:sigma-torus-knots}
  Let \(q\) be an odd prime.
  \begin{enumerate}
  \item\label{item:sigma-inv} For any \(1 \leq a \leq q-1\) we have
    \[\sigma(T(2,q),\chi_{a}) = -q + \frac{2a(q-a)}{q}, \quad \eta(T(2,q),\chi_{a}) = 0.\]
    For \(a=0\) we have
    \[\sigma(T(2,q),\chi_{0}) = \eta(T(2,q),\chi_{0}) = 0.\]
  \item\label{item:cabling} For any knot \(K\), and any \(1 \leq a \leq q-1\), we have
    \[\sigma(K(2,q),\chi_{a}) = -q + \frac{2a(q-a)}{q} + 2 \sigma_{K}(\xi_{q}^{a})\]
    and
    \[\eta(K(2,q),\chi_{a}) = 2 \eta_{K}(\xi_{q}^{a}),\]
    where \(\sigma_{K}\) and \(\eta_{K}\) denote the Levine-Tristram signature of \(K\) and nullity of \(K\), respectively.
  \end{enumerate}
\end{lemma}
\begin{proof}
  Consider first Item~\ref{item:sigma-inv}.
  Using the surgery formula~\cite[Lemma 3.1]{CassonGordon1} we obtain the desired formula for \(a \neq 0\).
  For \(a = 0\), we apply Item~\ref{item:properties-cs-invariants-1} of Lemma~\ref{prop:properties-cs-invariants}.

  The formula for the \(\sigma\)-invariant in Item~\ref{item:cabling} follows from~\cite[Corollary A.6]{ConwayNagel}.
  The formula for the nullity follows from~\cite[Proposition A.10]{ConwayNagel}.
  Although, the formula in~\cite[Proposition A.10]{ConwayNagel} is proved under the assumption that the winding number \(w\) is zero, the proof works without modification under the assumption that we work with double branched covers and \(w \equiv 0 \pmod{2}\).
\end{proof}

\section{Computer-aided calculations of Casson-Gordon invariants}
\label{sec:computations}

The purpose of this section is to explain our approach towards verification of Gilmer's criterion (Theorem~\ref{thm:genus-bounds}) for the knot from Theorem~\ref{thm:main-theorem}.
The main result of this section is Lemma~\ref{lemma:calculations}, whose proof was done with the aid of a computer.

Let \(K\) be the knot from Theorem~\ref{thm:main-theorem}.
Recall from Section~\ref{sec:link-forms-char}, that \(H_{1}(\Sigma(K))_{q}\) denotes the \(q\)-primary summand of \(H_{1}(\Sigma(K))\), for a prime \(q\).
By Lemma~\ref{lemma:character-vector-correspondence}, to any \(x \in H_{1}(\Sigma(K))_{q}\) we associate the character
\[\chi_{x} \colon H_{1}(\Sigma(K)) \to C_{q}, \quad \chi_{x}(y) = \exp\left(2 \pi i \lambda_{K}(x,y)\right).\]
The main result of this section is the following lemma.

\begin{lemma}\label{lemma:calculations}
  Let \(K\) denote the knot from Theorem~\ref{thm:main-theorem}.
  Fix \(q=83,103\).
  Let \(x \in H_{1}(\Sigma(K))_{q} \setminus \{0\}\) be an isotropic element, i.e., \(\lambda_{K}(x,x)=0\).
  There exists \(0<k<q\) such that for \(y = kx\) the following inequality is satisfied
  \[\left| \sigma(K,\chi_{y})  \right| > 5 + \eta(K,\chi_{y}). \]
\end{lemma}

The strategy of the proof is to first translate Lemma~\ref{lemma:calculations} into a purely algebraic problem, which can be solved by a computer.

\begin{problem}\label{problem:algebraic-problem}
  Fix a prime number \(q=83,103\) and let \(t\) be a positive integer relatively prime to~\(q\).
  For~\(V_{q} = \Z_{q}^{4}\) consider the following functions
  \[s_{q,t} \colon \Z_{q} \to \Q, \quad H_{q},\Sigma_{q} \colon V_{q} \to \Q, \quad Q_{q} \colon V_{q} \times V_{q} \to \Z_{q}\] given by
  \begin{itemize}
    \item
    \(s_{q,t}(x) =
    \begin{cases}
      -q + \frac{2x(q-x)}{q} + 2 \sigma_{T(2,t)}(\xi_{q}^{x}), & 0 < x < q, \\
      0, & x=0.
    \end{cases}\)
    \item \(\Sigma_{q}(x_{1},x_{2},x_{3},x_{4}) =
    \begin{cases}
      s_{83,17}(x_{1})-s_{83,11}(x_{2})-s_{83,13}(x_{3})+s_{83,1}(x_{4}), & q = 83, \\
      s_{103,11}(x_{1})-s_{103,17}(x_{2})-s_{103,17}(x_{3})+s_{103,13}(x_{4}), & q=103.
    \end{cases}
    \)
    \item \(Q_{q}(x,y) = x_{1}y_{1}-x_{2}y_{2}-x_{3}y_{3}+x_{4}y_{4}\),
  \end{itemize}
  where \(\sigma_{T(2,t)}\) denotes the Levine-Tristram signature function of the torus knot~\(T(2,t)\) and \(\xi_{q} = \exp\left( \frac{2 \pi i}{q} \right)\).
  Define
  \[L_{q} = \{x \in V_{q} \setminus \{0\} \colon Q_{q}(x,x) = 0\},\]
  and for any \(x \in L_{q}\) let
  \[M_{q}(x) = \{k \colon \left| \Sigma_{q}(kx) \right| > 8, \,\, 1 \leq k \leq q-1\}.\]
  Verify that for any \(x \in L_{q}\), \(M_{q}(x) \neq \emptyset\).
\end{problem}

\begin{lemma}\label{lemma:reformulation-CS-computation}
  Existence of a solution to Problem~\ref{problem:algebraic-problem}, for \(q=83\) and \(q=103\), implies that Lemma~\ref{lemma:calculations} holds.
\end{lemma}
\begin{proof}
  Lemma~\ref{lemma:branched-covers-cable-knots} implies that
  \[H_{1}(\Sigma(K))_{q} \cong
    \begin{cases}
      V_{q}, & q=83,103, \\
      0, & \text{otherwise}.
    \end{cases}
  \]
  Above we choose the identification \(H_{1}(\Sigma(K))_{q} \cong V_{q}\), for \(q=83,103\) such that the distinguished generators (see Item~2 of Lemma~\ref{lemma:branched-covers-cable-knots}) from \(H_{1}\) of
  double branched covers of the summands of \(K\) form a basis.
  Consider the primary decomposition of the linking form \(\lambda_{k}\), as in~\eqref{eq:primary-decomposition-linking-form},
  \[(H_{1}(\Sigma(K)),\lambda_{K}) \cong \bigoplus_{q \in \{83,103\}} (H_{1}(\Sigma(K))_{q},\lambda_{K,q}).\]
  With respect to this choice of basis, the primary summands \(\lambda_{K,q}\) can be identified with
  \[\lambda_{K,q} \cong
    \begin{cases}
      \ell_{q} \oplus -\ell_{q}, & q=83,103, \\
      0, & \text{otherwise},
    \end{cases}
  \]
  where
  \[\ell_{q} = (1/q) \oplus (-1/q),\]
  refer to Example~\ref{example:linking-forms-cyclic-groups} for the notation.
  In particular, under this identification, if \(x,y \in V_{q}\), then \(\lambda_{K}(x,y) = \lambda_{K,q}(x,y) = \frac{Q_{q}(x,y)}{q}\), for \(q=83,103\).
  Furthermore, by combining Proposition~\ref{prop:properties-cs-invariants} and Lemma~\ref{lemma:sigma-torus-knots}, we can see that for \(x \in V_{q}\) we have
  \[\sigma(K,\chi_{x}) = \Sigma_{q}(x), \quad \eta(K,\chi_{x}) = \# \{i \colon x_{i} \neq 0\},\]
  for \(x = (x_{1},x_{2},x_{3},x_{4})\).
  In particular, since \(x \neq 0\), it follows that
  \[\left| \eta(K,\chi_{x}) \right| \leq 3.\]
  Consequently, if \(k \in M_{q}(x)\), it follows that
  \[\left| \Sigma_{q}(kx) \right| > 8 \geq 5 + \eta(K,\chi_{kx}),\]
  hence \(y=kx\) satisfies the inequality in Lemma~\ref{lemma:calculations}.
  This concludes the proof.
\end{proof}




\begin{solution}[Computer-aided solution to Problem~\ref{problem:algebraic-problem}]
  The solution of the problem implemented in~\cite{sagecode} is basically a brute-force search with some optimizations.

  Fix \(q=83,103\).
  Let \(U(\Z_{q})\) denote the multiplicative group of units of \(\Z_{q}\).
  Observe that \(U(\Z_{q})\) acts on \(V_{q} \setminus \{0\}\) by multiplication.
  Denote by \(W_{q}\) the quotient of \(V_{q} \setminus \{0\}\) by the action of \(U(\Z_{q})\).
  Furthermore, for any \(x \in V_{q} \setminus \{0\}\), we will denote by \([x]\) its image under the canonical projection to \(W_{q}\).
  If \(y \in V_{q} \setminus \{0\}\), we say that \(y\) is in a \emph{canonical form} if the first nonzero coordinate of \(y\) is equal to one.
  If \(x \in L_{q}\) and \(y \in [x]\) is in canonical form, we say that \(y\) is a canonical representative of \([x]\).
  We make the following observations.
  \begin{enumerate}[(i)]
    \item If \(x \in L_{q}\), then the whole \(U(\Z_{q})\)-orbit of \(x\) belongs to \(L_{q}\).
    \item If \(x \in L_{q}\), then for any \(y\) in \([x]\), \(M_{q}(x)\) and \(M_{q}(y)\) are of the same cardinality.
    Indeed, write \(y = kx\), for some \(1 \leq k \leq q-1\), then
    \(M_{q}(y) = \{k^{-1} \cdot l \colon l \in M_{q}(x)\}\).
    \item For any \(x \in V_{q} \setminus \{0\}\), there exists a unique canonical representative \(x_{can} \in [x]\).
    \item Every vector \(x\) in a canonical form can be written uniquely in the form
    \begin{equation}
      \label{eq:canonical forms}
      x = (\underbrace{0,\ldots,0}_{3-k},1,x),
    \end{equation}
    where \(k=0,1,2,3\) and \(x \in V_{q}^{k}\) is arbitrary.
  \end{enumerate}
  The above observations suggest the following simplification: it is sufficient to check the condition in Problem~\ref{problem:algebraic-problem} only for vectors in canonical form.


  Hence, below we present a sketch of an algorithm for solving Problem~\ref{problem:algebraic-problem}.
  \begin{algorithm}
    \caption{The algorithm from~\cite{sagecode} solving Problem~\ref{problem:algebraic-problem}}
    \begin{algorithmic}[1]
      \For{$q \in\{83,103\}$}
      \For{$k=0,1,2,3$}
      \ForAll{$y \in V_{q}^{k}$} \label{alg:go-back-here}
      \State \(x \gets (\underbrace{0,\ldots,0}_{3-k},1,y)\) \Comment{Here we are generating all canonical representatives using~\eqref{eq:canonical forms}.}
      \If{\(Q_{q}(x,x) = 0\) in \(\Z_{q}\)}
      \For{\(l \in \{1,2,\ldots,q-1\}\)}
      \If{\(\left| \Sigma_{q}(lx) \right| > 8\)}
      \State Print \(x\) and \(l\) and go back to line~\ref{alg:go-back-here}.
      \EndIf
      \EndFor
      \State Print \(x\) and stop the program, since \(M_{q}(x) = \emptyset\).
      \EndIf
      \EndFor
      \EndFor
      \EndFor
    \end{algorithmic}
  \end{algorithm}

\end{solution}

\section{Genus bounds}
\label{sec:genus-bounds}

The purpose of this section is to give proof of Theorem~\ref{thm:main-theorem}.
Let us start with a lemma giving an upper bound for the four-genus of a certain infinite family of knots.

Let \(p_{1},p_{2},q_{1},q_{2},q_{3} \in \Z_{+}\) be odd primes and consider the knot
\begin{align*}
  K(p_{1},p_{2},q_{1},q_{2},q_{3}) &= 
  T(2,q_{1};2,p_{1}) \ \# \ -T(2,q_{2};2,p_{1}) \ \# \ 
  T(2,p_{1}) \ \# \ -T(2,q_{3};2,p_{1}) \\
     &\# \ T(2,q_{2};2,p_{2}) \ \# \ -T(2,p_{2}) \ \#  \ T(2,q_{3};2,p_{2}) \ \# \ -T(2,q_{1};2,p_{2}).
\end{align*}
Observe that taking \(p_{1}=83,p_{2}=103,q_{1}=17,q_{2}=11,q_{3}=13,\) we recover the knot \(K\) from Theorem~\ref{thm:main-theorem}.

\begin{lemma}\label{lemma:upper-bound}
  For any choice of positive odd integers~\(p_{1},p_{2}\) and non-negative odd integers~\(q_{1},q_{2},q_{3}\), we have
  \[g_{4}(K(p_{1},p_{2},q_{1},q_{2},q_{3})) \leq 2.\]
\end{lemma}


\begin{figure}
  \fontsize{8}{5}\selectfont
  \centering{

    \resizebox{0.9\textwidth}{!}{
      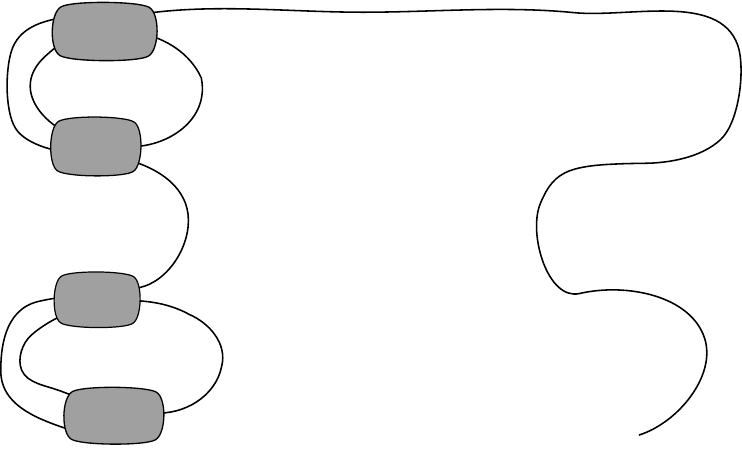}
  }
  \caption{Band moves (indicated in red) giving a cobordism from the knot \(J(r_{1},r_{2},s_{1},s_{2})\) to the \((2,0)\)-cable of the knot \(J'\), see~\eqref{eq:knot-J-prime}.}
  \label{fig:ribbon-moves}
\end{figure}

\begin{proof}
  Consider the knot
  \[J(r_{1},r_{2},s_{1},s_{2}) = T(2,s_{1};2,r_{2}) \ \# \ -T(2,s_{1};2,r_{1}) \ \# \ T(2,s_{2};2,r_{2}) \ \# \ -T(2,s_{2};2,r _{1}),\]
  where \(r_{1},r_{2},s_{1},s_{2}\) are odd integers such that \(r_{1},r_{2}>2\)
  Observe, than we can write
  \[K(p_{1},p_{2},q_{1},q_{2},q_{3}) = J_{1}(p_{1},p_{2},q_{1},0) \# J_{2}(p_{2},p_{1},q_{2},q_{3}).\]
  We will prove that \(g_{4}(J(r_{1},r_{2},s_{1},s_{2})) \leq 1\).
  This would imply that
  \[g_{4}(K(p_{1},p_{2},q_{1},q_{2},q_{3})) \leq g_{4}(J_{1}(p_{1},p_{2},q_{1},0)) + g_{4}(J_{2}(p_{2},p_{1},q_{2},q_{3})) \leq 2,\]
  as desired.

  To prove that \(g_{4}(J(r_{1},r_{2},s_{1},s_{2})) \leq 1\), we will use a modification of the argument from~\cite[Section 6]{HeddenKirkLivingston}.
  We will construct an explicit genus-one null-cobordism for \(J := J(r_{1},r_{2},s_{1},s_{2})\) in \(D^{4}\).
  Start by performing on~\(J\) band moves depicted in Figure~\ref{fig:ribbon-moves}.
  As a result, we obtain a genus-one cobordism \(S_{1}\) from \(J\) to the \(2\)-component link \(L_{0}\), which is isotopic to the \((2,0)\)-cable of the knot
  \begin{equation}
    \label{eq:knot-J-prime}
    J' = T(2,r_{2}) \# -T(2,r_{2}) \# T(2,r_{1}) \# -T(2,r_{1}).
\end{equation}
  Since \(J'\) is a slice knot (in fact a ribbon knot), let \(S'\) be the concordance from~\(J'\) to the unknot.
  Let \(S_{2}\) be the \((2,0)\)-cable of \(S'\), i.e., \(S_{2}\) is the union of two-parallel copies of \(S'\) contained in some small tubular neighborhood of \(S'\).
  The surface \(S_{2}\), diffeomorphic to a disjoint union of two annuli, is a concordance from \(L_{0}\) to the two-component unlink.
  By joining~\(S_{1}\) and~\(S_{2}\) along~\(L_{0}\) we obtain a connected, genus-one cobordism \(S\) from \(J\) to the two-component unlink. By capping off \(S\) we obtain the desired null-cobordism for \(J\).
     
\end{proof}


We are now ready to prove Theorem~\ref{thm:main-theorem}.
\begin{proof}[Proof of Theorem~\ref{thm:main-theorem}]
  Observe that by~\cite[Theorem 1.1 and Proposition 5.4]{conway2019nonslice}, the knot \(K:=K(83,103,17,11,13)\) from Theorem~\ref{thm:main-theorem} is algebraically slice but not slice.
  The crucial point here, is that the primes $p_1=83$, $p_2=103$, $q_1=11$, $q_2=13$, $q_3=17$ are pairwise distinct.
  By Lemma~\ref{lemma:upper-bound}, \(1 \leq g_{4}^{top}(K) \leq g_{4}(K) \leq 2\).

  
  Suppose that \(g_{4}^{top}(K)=1\), we will show that this assumption contradicts Lemma~\ref{lemma:calculations}.
  Before delving into the proof, recall from Section~\ref{sec:computations} that
  \begin{equation}
    \label{eq:H1-and-linking-form}
    H_{1}(\Sigma(K)) = \Z_{83}^{4} \oplus \Z_{103}^{4}, \quad \lambda_{K} \cong (\ell_{83} \oplus -\ell_{83}) \oplus (\ell_{103} \oplus -\ell_{103}),
  \end{equation}
  where \(\ell_{q} = \left( 1/q \right) \oplus \left( -1/q \right)\).
  Since \(K\) is algebraically slice, the signature of \(K\) is zero, hence by Theorem~\ref{thm:genus-bounds}, there is a decomposition \(\lambda_{K} = \beta_{1} \oplus \beta_{2}\), where
  \begin{itemize}
  \item \(\beta_{1}\) is represented by a square integral matrix \(A\) of size \(2 \times 2\) and signature zero,
  \item \(\beta_{2}\) is metabolic, and there is a metabolizer \(L\) for \(\beta_{2}\), such that for any \(x \in L\), we have
    \begin{equation}
      \label{eq:Gilmer-condition}
      \left| \sigma(K,\chi_{x}) \right| \leq 5 + \eta(K,\chi_{x}).
    \end{equation}
  \end{itemize}

  We will split the argument into two cases: either \(\beta_{2} \neq 0\), or \(\beta_{2} = 0\).
  Let us first consider the case \(\beta_{2} \neq 0\).
  Since \(L\) is a metabolizer for \(\beta_{2}\), by Lemma~\ref{lemma:primary-decomposition-metabolizers}, it can be written in the form \(L = L_{83} \times L_{103}\), where \(L_{p}\) is a metabolizer for the \(p\)-primary part \((\beta_{2})_{p}\) of \(\beta_{2}\), for \(p=83,103\).
  Since \(\beta_{2} \neq 0\), we can assume without loss of generality, that \((\beta_{2})_{83} \neq 0\).
  In particular, \(L_{83} \neq 0\).
  Fix any \(x \in L_{83} \setminus \{0\}\), Lemma~\ref{lemma:calculations} shows that for some multiple \(y = kx \in L_{83}\), with \(1 \leq k < 83\), the condition~\eqref{eq:Gilmer-condition} is not satisfied.
  Hence, we reached a contradiction, which shows that the condition  \(\beta_{2} \neq 0\) cannot be satisfied.
  
  If \(\beta_{2} = 0\), then \(\lambda_{K}\) is represented by a symmetric integral matrix \(A\) of size \(2 \times 2\) and signature zero.
  In particular, \(A\) is a presentation matrix for \(H_{1}(\Sigma(K))\), i.e.,
  \[H_{1}(\Sigma(K)) \cong \Z^{2} / A\Z^{2}.\]
  However, the above isomorphism is not possible, since the minimal number of generators for \(H_{1}(\Sigma(K))\) is four, see~\eqref{eq:H1-and-linking-form}.
  This shows that neither condition \(\beta_{2} \neq 0\) or \(\beta_{2} = 0\) can hold.
  Consequently, our assumption \(g_{4}^{top}(K) = 1\) lead us to a contradiction, hence \(g_{4}^{top}(K) = g_{4}(K) = 2\) and the theorem is proved.
\end{proof}

\begin{remark}\label{remark:optimizations}
  In principle, verifying assumptions of Theorem~\ref{thm:genus-bounds} requires to check all isotropic elements of \(H_{1}(\Sigma(K))\).  Since the order of \(H_{1}(\Sigma(K))\) is \(N=83 \cdot 103 = 8549\),
  there are roughly \(N^{3} \approx 6 \cdot 10^{12}\) non-zero isotropic elements in \(H_{1}(\Sigma(K))\).
  Indeed, according to~\cite[Chapter 3.7.2]{wilsonFiniteSimpleGroups2009}, the \(q\)-primary summand of \(H_{1}(\Sigma(K))\) contains roughly~\(q^{3}\) isotropic vectors, for \(q=83,103\).
  Furthermore, by modifying slighlty the argument from the proof of Lemma~\ref{lemma:primary-decomposition-metabolizers}, we can see that every isotropic element~\(x\) of \(H_{1}(\Sigma(K))\) can be written as
  a sum~\(x = x_{83}+x_{103}\), where, for \(q=83,103\),~\(x_{q}\) is isotropic and belongs to \(H_{1}(\Sigma(K))_{q}\).

  Lemma~\ref{lemma:calculations} and the proof of Theorem~\ref{thm:main-theorem} show that instead of checking all isotropic vectors, it is sufficient to check only isotropic vectors belonging to primary
  summands of the linking form.
  Restricting to primary parts, reduces the number of possibilities to check roughly \(83^{3} + 103^{3} \approx 1.6 \cdot 10^{6}\).
\end{remark}

\bibliographystyle{amsalpha}
\bibliography{BiblioAlgebraicKnots}

\providecommand{\bysame}{\leavevmode\hbox to3em{\hrulefill}\thinspace}
\providecommand{\MR}{\relax\ifhmode\unskip\space\fi MR }
\providecommand{\MRhref}[2]{%
  \href{http://www.ams.org/mathscinet-getitem?mr=#1}{#2}
}
\providecommand{\href}[2]{#2}
\begin{thebibliography}{FNOP20}

\bibitem[Bro82]{brown}
Kenneth~S. Brown, \emph{Cohomology of {Groups}}, Graduate {Texts} in
  {Mathematics}, vol.~87, Springer New York, New York, NY, 1982 (en).

\bibitem[BZH14]{Burde-Zieschang}
Gerhard Burde, Heiner Zieschang, and Michael Heusener, \emph{Knots}, 3rd fully
  revised and extented edition ed., De Gruyter Stud. Math., vol.~5, Berlin:
  Walter de Gruyter, 2014.

\bibitem[CFH16]{Conway-Friedl-Gerrit}
Anthony Conway, Stefan Friedl, and Gerrit Herrmann, \emph{Linking forms
  revisited}, Pure Appl. Math. Q. \textbf{12} (2016), no.~4, 493--515
  (English).

\bibitem[CG78]{CassonGordon1}
Andrew~J. Casson and Cameron~McA. Gordon, \emph{On slice knots in dimension
  three}, Algebraic and geometric topology \textup{(}{P}roc. {S}ympos. {P}ure
  {M}ath., {S}tanford {U}niv., {S}tanford, {C}alif., 1976\textup{)}, {P}art 2,
  Proc. Sympos. Pure Math., XXXII, Amer. Math. Soc., Providence, R.I., 1978,
  pp.~39--53.

\bibitem[CG86]{CassonGordon2}
\bysame, \emph{Cobordism of classical knots}, \`A la recherche de la topologie
  perdue, Progr. Math., vol.~62, Birkh{\"a}user Boston, Boston, MA, 1986, With
  an appendix by P. M. Gilmer, pp.~181--199.

\bibitem[CKP19]{conway2019nonslice}
Anthony Conway, Min~Hoon Kim, and Wojciech Politarczyk, \emph{Non-slice linear
  combinations of iterated torus knots}, 2019,
  \href{https://arxiv.org/abs/1910.01368}{arXiv:1910.01368}, to appear in AGT.

\bibitem[CN20]{ConwayNagel}
Anthony Conway and Matthias Nagel, \emph{Stably slice disks of links}, J.
  Topol. \textbf{13} (2020), no.~3, 1261--1301.

\bibitem[Con]{conway-CG-notes}
Anthony Conway, \emph{Algebraic concordance and {C}asson-{G}ordon invariants},
  available online:
  \url{https://drive.google.com/file/d/1hC5Bq1GxUmDMR7VGkJmTT_QnuLq-t9vt/view},
  p.~26.

\bibitem[DK01]{davis-kirk}
James~F. Davis and P.~Kirk, \emph{Lecture notes in algebraic topology},
  Graduate studies in mathematics, no. v. 35, American Mathematical Society,
  Providence, R.I, 2001.

\bibitem[EN85]{EN}
David Eisenbud and Walter Neumann, \emph{Three-dimensional link theory and
  invariants of plane curve singularities}, Annals of Mathematics Studies, vol.
  110, Princeton University Press, Princeton, NJ, 1985.

\bibitem[FG03]{FlorensGilmer}
Vincent Florens and Patrick~M. Gilmer, \emph{On the slice genus of links},
  Algebr. Geom. Topol. \textbf{3} (2003), 905--920.

\bibitem[FNOP20]{friedl2020survey}
Stefan Friedl, Matthias Nagel, Patrick Orson, and Mark Powell, \emph{A survey
  of the foundations of four-manifold theory in the topological category},
  2020, \href{https://arxiv.org/abs/1910.07372}{arXiv:1910.07372}.

\bibitem[FQ90]{Freedman-Quinn}
Michael~H. Freedman and Frank Quinn, \emph{Topology of 4-manifolds}, Princeton
  Mathematical Series, vol.~39, Princeton University Press, Princeton, NJ,
  1990.

\bibitem[Gil81]{gilmerTopologicalProofSignature1981}
Patrick Gilmer, \emph{Topological proof of the {$G$}-signature theorem for
  {$G$} finite.}, Pacific Journal of Mathematics \textbf{97} (1981), no.~1,
  105--114.

\bibitem[Gil82]{gilmerSliceGenusKnots1982a}
Patrick~M. Gilmer, \emph{On the slice genus of knots}, Invent Math \textbf{66}
  (1982), no.~2, 191--197.

\bibitem[HKL12]{HeddenKirkLivingston}
Matthew Hedden, Paul Kirk, and Charles Livingston, \emph{Non-slice linear
  combinations of algebraic knots}, J. Eur. Math. Soc. \textup{(}JEMS\textup{)}
  \textbf{14} (2012), no.~4, 1181--1208.

\bibitem[KK80]{KawauchiKojima}
Akio Kawauchi and Sadayoshi Kojima, \emph{Algebraic classification of linking
  pairings on {$3$}-manifolds}, Math. Ann. \textbf{253} (1980), no.~1, 29--42.

\bibitem[KS77]{kirby-siebenmann}
Robion~C. Kirby and L.~Siebenmann, \emph{Foundational essays on topological
  manifolds, smoothings, and triangulations}, Annals of mathematics studies,
  no. no. 88, Princeton University Press, Princeton, N.J, 1977.

\bibitem[Kup17]{kupers}
Alexander Kupers, \emph{Three lecture on topological manifolds}, 2017,
  \url{https://www.utsc.utoronto.ca/people/kupers/wp-content/uploads/sites/50/2021/01/toplectures.pdf}.

\bibitem[Lew14]{Lewark}
Lukas Lewark, \emph{Rasmussen's spectral sequences and the {$sl_N$}-concordance
  invariants}, Adv. Math. \textbf{260} (2014), 59--83.

\bibitem[Lit79]{Litherland-signature}
Richard~A. Litherland, \emph{Signatures of iterated torus knots}, Topology of
  low-dimensional manifolds \textup{(}{P}roc. {S}econd {S}ussex {C}onf.,
  {C}helwood {G}ate, 1977\textup{)}, Lecture Notes in Math., vol. 722,
  Springer, Berlin, 1979, pp.~71--84.

\bibitem[Lit84]{Litherland}
\bysame, \emph{Cobordism of satellite knots}, Four-manifold theory
  \textup{(}{D}urham, {N}.{H}., 1982, Contemp. Math., vol.~35, Amer. Math.
  Soc., Providence, RI, 1984, pp.~327--362.

\bibitem[LL16]{Lewar-Lobb}
Lukas Lewark and Andrew Lobb, \emph{New quantum obstructions to sliceness},
  Proc. Lond. Math. Soc. (3) \textbf{112} (2016), no.~1, 81--114.

\bibitem[LM83]{LivingstonMelvinAlgebraicKnots}
Charles Livingston and Paul Melvin, \emph{Algebraic knots are algebraically
  dependent}, Proc. Amer. Math. Soc. \textbf{87} (1983), no.~1, 179--180.

\bibitem[Mar21]{sagecode}
Maria Marchwicka, \emph{signature\_function},
  \url{https://git.wmi.amu.edu.pl/marchwicka/signature_function}, 2021, Sage
  script which computes untwisted and twisted signatures of iterated torus
  knots.

\bibitem[Miy94]{Miyazaki}
Katura Miyazaki, \emph{Nonsimple, ribbon fibered knots}, Trans. Amer. Math.
  Soc. \textbf{341} (1994), no.~1, 1--44.

\bibitem[NP17]{Nagel-Powell}
Matthias Nagel and Mark Powell, \emph{Concordance invariance of
  {Levine}-{Tristram} signatures of links}, Doc. Math. \textbf{22} (2017),
  25--43 (English).

\bibitem[OSS17]{OSS}
Peter~S. Ozsv\'{a}th, Andr\'{a}s~I. Stipsicz, and Zolt\'{a}n Szab\'{o},
  \emph{Concordance homomorphisms from knot {F}loer homology}, Adv. Math.
  \textbf{315} (2017), 366--426.

\bibitem[Rud76]{Rudolph}
Lee Rudolph, \emph{How independent are the knot-cobordism classes of links of
  plane curve singularities\textup{?}}, Notices Amer. Math. Soc. \textbf{23}
  (1976), 410.

\bibitem[{Sag}21]{sagemath}
{Sage Developers}, \emph{{S}agemath, the {S}age {M}athematics {S}oftware
  {S}ystem ({V}ersion 3.8.5)}, 2021, {\tt https://www.sagemath.org}.

\bibitem[Sto68]{stong}
Robert~E. Stong, \emph{Notes on cobordism theory. {Preliminary} informal notes
  of {University} courses and seminars in mathematics}, Math. Notes
  (Princeton), Princeton University Press, Princeton, NJ, 1968 (English).

\bibitem[Tan23]{Tange}
Motoo Tange, \emph{Upsilon invariants of {L}-space cable knots}, Topology and
  its Applications \textbf{324} (2023), 108335.

\bibitem[Vir09]{viro}
Oleg Viro, \emph{Twisted acyclicity of a circle and signatures of a link},
  Journal of Knot Theory and Its Ramifications \textbf{18} (2009), no.~06,
  729--755.

\bibitem[Wal63]{WallQuadratic}
C.~T.~C. Wall, \emph{Quadratic forms on finite groups, and related topics},
  Topology \textbf{2} (1963), 281--298.

\bibitem[Wil09]{wilsonFiniteSimpleGroups2009}
Robert~A. Wilson, \emph{The {Finite} {Simple} {Groups}}, Graduate {Texts} in
  {Mathematics}, vol. 251, Springer London, London, 2009.

\end{thebibliography}
\end{document}